\documentclass[a4paper,11pt]{amsart}

\usepackage[T1,T2A]{fontenc}
\usepackage[english,russian]{babel}
\usepackage{amsmath}
\usepackage{amsfonts}
\usepackage{amssymb}
\usepackage{mathrsfs}
\usepackage{amsthm}
\usepackage{graphicx}
\usepackage{color}

\newcommand{\df}{\buildrel\mathrm{def}\over=}

\newcommand{\RR}{\mathrm R}
\newcommand{\LL}{\mathrm L}
\newcommand{\eps}{\varepsilon}

\newcommand{\half}{\tfrac12}
\newcommand{\eq}[2]{\begin{equation}\label{#1} #2 \end{equation}}

\newtheorem{Le}{Lemma}[section]
\newtheorem{Prop}[Le]{Proposition}
\newtheorem{Th}[Le]{Theorem}

\theoremstyle{definition}
\newtheorem{Def}[Le]{Definition}
\newtheorem{Rem}[Le]{Remark}

\numberwithin{equation}{section}

\newcommand{\ti}[1]{_{\scriptscriptstyle\text{\rm #1}}}
\newcommand{\ut}[1]{^{\scriptscriptstyle\text{\rm #1}}}

\newcommand{\Oe}{\Omega_\eps}
\newcommand{\Om}[1]{\Omega\ut{#1}_\eps}

\graphicspath{{./Pictures/}}

\newcommand\mytitle{Some extremal problems for martingale transforms. II}
\newcommand\mytitleshort{Extremal problems for martingale transforms}

\begin{document}
\title[\mytitleshort]{\mytitle}

\author{V.~Vasyunin}


\begin{abstract}
This paper is a direct continuation of the paper~\cite{First}. 
By this reason neither introductory part of the paper nor the 
list of references are not duplicated. However for the reader 
convenience, the formulas from the first paper that are cited 
here are collected in a special addendum at the end of the 
paper with their original numbers.

At this paper two new local foliations are investigated:
minor pockets and rectangles. The appearance of such local 
foliations is illustrated by the further investigation of 
the examples with the boundary functions being the polynomials 
of the third degree.
\end{abstract}

\thanks{This work is supported by the Russian Science 
Foundation, grant number 19-71-10023 
(https://rscf.ru/project/19-71-10023/).}

\maketitle

\setcounter{section}{7}
\setcounter{figure}{16}

\section{Minor pockets}

In this section we consider the horizontal herringbones with 
a spine going from a point on a boundary line to another 
point on the same line and not intersecting the middle 
line of the strip. We will refer to such foliation as a 
{\it minor pocket} to separate them from {\it major pockets}. 
In a case of a major pocket the spine of a herringbone 
also starts and ends on the same boundary line but intersects 
the middle line of the strip twice. Such pockets have 
completely different origin and they will be considered later.

Minor pockets appear inside the domains with the simple 
foliation. Due to the symmetry we can consider only the 
domains with the right simple foliation and the pockets that 
appear on the upper boundary. Any other cases can be obtained 
by means of change the sign of the variables $x_1$ and/or $x_2$.

Let us assume that for $\eps\in[\eps_1,\eps_0)$ we have 
a standard right simple candidate on a domain 
$\Om{R}[u_1,u_2]$. This means that for $u\in[u_1,u_2]$ and 
$\eps\in[\eps_1,\eps_0)$ two conditions
\begin{align}
\label{150501}
2\eps D_+(u,\eps)=f_+'(u+\eps)-f_-'(u-\eps)-
2\eps f_+''(u+\eps)&>0\,;
\\
\label{150502}
2\eps D_-(u,\eps)=f_+'(u+\eps)-f_-'(u-\eps)-
2\eps f_-''(u-\eps)&>0
\end{align}
are fulfilled (see conditions (2.2) in~\cite{First}). 
Moreover, for $\eps=\eps_0$ one of these conditions 
degenerates. We will assume that this is 
condition~\eqref{150501}, i.\,e., there exists a point 
$u_0\in(u_1,u_2)$ such that
\eq{150503}{
D_+(u_0,\eps_0)=0\,,
}
but for all other points $u\in[u_1,u_0)\cup(u_0,u_2]$ 
condition~\eqref{150501} still holds for $\eps=\eps_0$. 
We will assume that this root is not degenerated if the 
following sense: the function $\eps\mapsto D_+(u_0,\eps)$ 
has a simple root at the point $\eps=\eps_0$, i.\,e.,
\eq{150505}{
\frac{\partial D_+}{\partial x_2}(u_0,\eps_0)
=\frac{f_-''(u_0-\eps_0)-f_+''(u_0+\eps_0)}{2\eps_0}-
f_+'''(u_0+\eps_0)<0\,,
}
and the function $u\mapsto D_+(u,\eps_0)$ has a root of 
the second order at the point 
$u=u_0$, i.\,e.,
\eq{150506}{
\frac{\partial D_+}{\partial x_1}(u_0,\eps_0)
=\frac{f_+''(u_0+\eps_0)-f_-''(u_0-\eps_0)}{2\eps_0}-
f_+'''(u_0+\eps_0)=0\,,
}
and
\eq{150507}{
\frac{\partial^2 D_+}{\partial x_1^2}(u_0,\eps_0)
=\frac{f_+'''(u_0+\eps_0)-f_-'''(u_0-\eps_0)}{2\eps_0}-
f_+''''(u_0+\eps_0)>0\,.
}
Let us mention that condition~\eqref{150506} allows us to 
rewrite~\eqref{150505} in the following simpler form:
\eq{120301}{
f_+'''(u_0+\eps_0)>0\,.
}

Under these assumptions the equation $D_+(u,\eps)=0$ has 
two roots $u^l(\eps)$ and $u^r(\eps)$ ($u^l<u^r$) for every 
$\eps$ in some right neighborhood of $\eps_0$. Writing $u_+$ 
in some assertion, we will mean that this assertion is true 
for each of these two roots, for example, $D_+(u_+,\eps)=0$. 
We assume that $\eps$ is sufficiently close to $\eps_0$, so 
that $D_+(u,\eps)>0$ for $u\in[u_1,u^l)\cup(u^r,u_2]$ and 
$D_+(u,\eps)<0$ for $u\in(u^l,u^r)$. All these assumptions 
mean that for $\eps$ sufficiently close to $\eps_0$ the 
function $u\mapsto D_+(u,\eps)$ decreases at the point $u^l$ 
and increases at the point $u^r$. Therefore,
$(D_+)'_u(u^l)<0$ and $(D_+)'_u(u^r)>0$, i.\,e.,
\eq{160501}{
f_+''(u^l+\eps)-f_-''(u^l-\eps)-2\eps f_+'''(u^l+\eps)<0
}
and
\eq{160502}{
f_+''(u^r+\eps)-f_-''(u^r-\eps)-2\eps f_+'''(u^r+\eps)>0\,.
}

We will assume that $\eps$ is so close to $\eps_0$ that we 
have not only~\eqref{120301}, but
\eq{160503}{
f_+'''(u_++\eps)>0
}
as well.

We need some new notation for minor pockets. The pockets 
appearing on the upper boundary we mark by the index ``$+$'' 
and the index ``$-$'' will be used for the pockets on the lower boundary. If it is a right 
herringbone, then it will be marked by the letter $\mathrm E$ (east),
for a left herringbone will be used the letter $\mathrm W$ (west).

Now we are ready to state the main assertion of this section.

\begin{Th}
\label{160504}
Let the boundary functions $f_\pm$ possess the properties described above\textup, i\,.e.\textup,
\begin{itemize}
\item $D_+(u,\eps)>0$ and $D_-(u,\eps)>0$ for $u\in[u_1,u_2]$ and $\eps\in(\eps_-,\eps_0),$ 
and therefore we are able to construct the right simple foliation on the domain $\Om{R}(u_1,u_2)$\textup;
\item $\eps_0$ is a simple root of the function $\eps\mapsto D_+(u_0,\eps)$ and $D_+(u,\eps_0)>0$
for $u\in[u_1,u_2]\setminus\{u_0\}$\textup;
\item $D_-(u,\eps_0)>0$ for $u\in[u_1,u_2]$\textup;
\item $u_0$ is a second order root of the function $u\mapsto D_+(u,\eps_0)$.
\end{itemize}
Then there exist two roots $u^l$ and $u^r$ $(u^l<u^r)$ of the equation $D_+(u,\eps)=0$
for $\eps\in(\eps_0,\eps_+)$ with some $\eps_+,$ and we can construct 
the following foliation \textup{(see~Fig.~\ref{101201})}
\eq{160505}{
\Om{R}(u_1,u^l)\cup\Om{W+}(u^l,u^r)\cup\Om{R}(u^r,u_2)\,,
}
where the foliation of the domain
$$
\Om{W+}(u^l,u^r)\df\big\{(x_1,x_2)\colon u^l\le x_1-x_2\le u^r,\;-\eps\le x_2\le\eps\big\}
$$
is a left horizontal herringbone extending from the point $(u^l+\eps,\,\eps)$ till 
the point $(u^r+\eps,\,\eps)$.
\end{Th}
\begin{figure}[h]
    \centering
    \includegraphics[scale = 0.85]{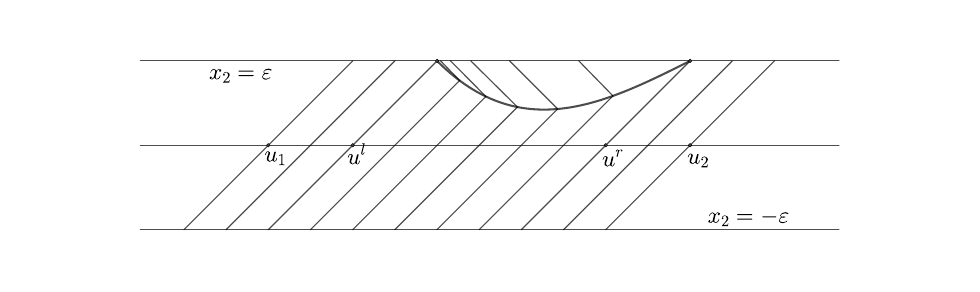}
    \caption{Foliation~\eqref{160505}}
    \label{101201}
\end{figure}

Now we state a symmetrical theorem describing what happens when the function $D_-$ 
vanishes on the upper boundary for some $\eps_0$.

\begin{Th}
\label{120302}
Let the boundary functions $f_\pm$ possess the following properties\textup:
\begin{itemize}
\item $D_+(u,\eps)>0$ and $D_-(u,\eps)>0$ for $u\in[u_1,u_2]$ and $\eps\in(\eps_-,\eps_0),$ 
and therefore we are able to construct the right simple foliation on the domain $\Om{R}(u_1,u_2)$\textup;
\item $\eps_0$ is a simple root of the function $\eps\mapsto D_-(u_0,\eps)$ and $D_-(u,\eps_0)>0$
for $u\in[u_1,u_2]\setminus\{u_0\}$\textup;
\item $D_+(u,\eps_0)>0$ for $u\in[u_1,u_2]$\textup;
\item $u_0$ is a second order root of the function $u\mapsto D_-(u,\eps_0)$.
\end{itemize}
Then there exist two roots $u^l$ and $u^r$ $(u^l<u^r)$ of the equation $D_-(u,\eps)=0$
for $\eps\in(\eps_0,\eps_+)$ with some $\eps_+,$ and we can construct 
the following foliation \textup{(see~Fig.~\ref{111201})}
\eq{120305}{
\Om{R}(u_1,u^l)\cup\Om{E-}(u^l,u^r)\cup\Om{R}(u^r,u_2)\,,
}
where the foliation of the domain
$$
\Om{E-}(u^l,u^r)=\big\{(x_1,x_2)\colon u^l\le x_1-x_2\le u^r,\;-\eps\le x_2\le\eps\big\}
$$
is a right horizontal herringbone extending from the point $(u^r\!-\eps,-\eps)$ till 
$(u^l\!-\eps,-\eps)$.
\end{Th}
\begin{figure}[h]

    \centering
    \includegraphics[scale = 0.85]{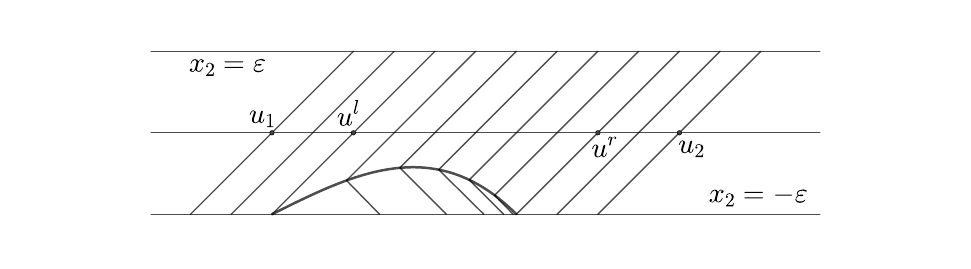}
    \caption{Foliation~\eqref{120305}}
    \label{111201}
\end{figure}

If we start with the left simple foliation, we obtain the following analogous propositions.

\begin{Th}
\label{120303}
Let the boundary functions $f_\pm$ possess the following properties\textup:
\begin{itemize}
\item $D_+(u,-\eps)>0$ and $D_-(u,-\eps)>0$ for $u\in[u_1,u_2]$ and $\eps\in(\eps_-,\eps_0),$ 
and therefore we are able to construct the left simple foliation on the domain $\Om{L}(u_1,u_2)$\textup;
\item $\eps_0$ is a simple root of the function $\eps\mapsto\! D_+(u_0,\!-\eps)$ and 
\hbox{$D_+(u,\!-\eps_0)>0$} for $u\in[u_1,u_2]\setminus\{u_0\}$\textup;
\item $D_-(u,-\eps_0)>0$ for $u\in[u_1,u_2]$\textup;
\item $u_0$ is a second order root of the function $u\mapsto D_+(u,-\eps_0)$.
\end{itemize}
Then there exist two roots $u^l$ and $u^r$ $(u^l\!<\!u^r)$ of the equation 
\hbox{$D_+(u,\!-\eps)=0$} for $\eps\in(\eps_0,\eps_+)$ with some $\eps_+,$ and we can 
construct the following foliation \textup{(see~Fig.~\ref{121201})}
\eq{120306}{
\Om{L}(u_1,u^l)\cup\Om{E+}(u^l,u^r)\cup\Om{L}(u^r,u_2)\,,
}
where the foliation of the domain
$$
\Om{E+}(u^l,u^r)=\big\{(x_1,x_2)\colon u^l\le x_1+x_2\le u^r,\;-\eps\le x_2\le\eps\big\}
$$
is a right horizontal herringbone extending from the point $(u^r-\eps,\,\eps)$ till 
$(u^l-\eps,\,\eps)$.
\end{Th}
\begin{figure}[h]
    \centering
    \includegraphics[scale = 0.9]{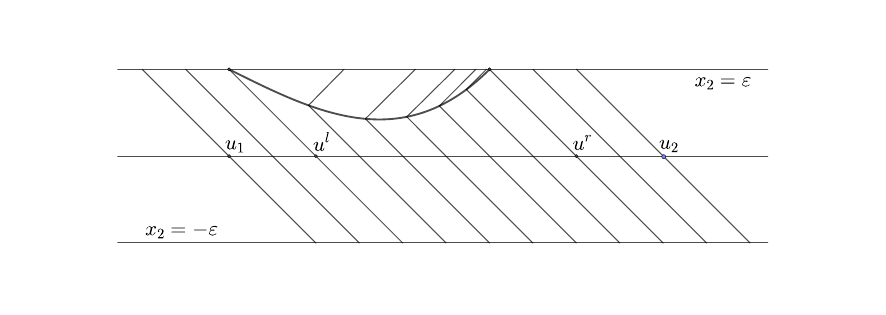}
    \caption{Foliation~\eqref{120306}}
    \label{121201}
\end{figure}

\begin{Th}
\label{120304}
Let the boundary functions $f_\pm$ possess the following properties\textup:
\begin{itemize}
\item $D_+(u,-\eps)>0$ and $D_-(u,-\eps)>0$ for $u\in[u_1,u_2]$ and $\eps\in(\eps_-,\eps_0),$ 
and therefore we are able to construct the left simple foliation on the domain $\Om{L}(u_1,u_2)$\textup;
\item $\eps_0$ is a simple root of the function $\eps\mapsto\! D_-(u_0,\!-\eps)$ and 
\hbox{$D_-(u,\!-\eps_0)>0$} for $u\in[u_1,u_2]\setminus\{u_0\}$\textup;
\item $D_+(u,\eps_0)>0$ for $u\in[u_1,u_2]$\textup;
\item $u_0$ is a second order root of the function $u\mapsto D_-(u,\eps_0)$.
\end{itemize}
Then there exist two roots $u^l$ and $u^r$ $(u^l<u^r)$ of the equation $D_-(u,\eps)=0$
for $\eps\in(\eps_0,\eps_+)$ with some $\eps_+,$ and we can construct 
the following foliation \textup{(see~Fig.~\ref{131201})}
\eq{120307}{
\Om{L}(u_1,u^l)\cup\Om{W-}(u^l,u^r)\cup\Om{L}(u^r,u_2)\,,
}
where the foliation of the domain
$$
\Om{W-}(u^l,u^r)=\big\{(x_1,x_2)\colon u^l\le x_1+x_2\le u^r,\;-\eps\le x_2\le\eps\big\}
$$
is a left horizontal herringbone extending from the point $(u^l+\eps,-\eps)$ till 
$(u^r+\eps,-\eps)$.
\end{Th}
\begin{figure}[h]
    \centering
    \includegraphics[scale = 0.95]{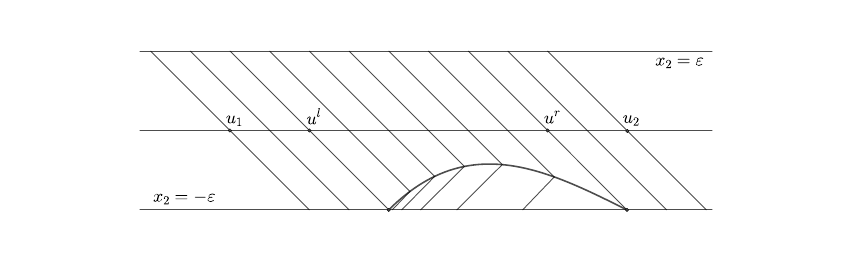}
    \caption{Foliation~\eqref{120307}}
    \label{131201}
\end{figure}

We will prove Theorem~\ref{160504} only. The other three theorems can be obtained by 
the symmetry. 

\begin{proof}[Proof of Theorem~\textup{\ref{160504}}] Let us begin with the investigation of the field~(5.3)  in a small 
domain including the stationary points $U^l\df(u^l,\eps)$ and $U^r\df(u^r,\eps)$. 
We start from consideration of three curves $X_0$, $X_1$, and $X_\infty$. We define $X_0$ as a set of 
points, where the integral curves of our field have the slope equal to $0$ (i.\,e., $\dot x_2=0$). On  $X_1$ the slope is equal to $1$ 
(i.\,e., $\dot x_2=\dot x_1$), and on $X_\infty$ the slope is equal to $\infty$ (i.\,e., $\dot x_1=0$), correspondingly.

Slightly abusing notation we will use the same symbols $X_i$ the following functions:
\begin{align*}
X_0(x_1,x_2)&=f_-'-f_+'-(\eps-x_2)f_-''+(\eps+x_2)f_+''\,;
\\
X_1(x_1,x_2)&=f_+'-f_-'-2x_2f_+''\,;
\\
X_\infty(x_1,x_2)&=\eps\big(f_+'-f_-'\big)-x_2(\eps-x_2)f_-''-x_2(\eps+x_2)f_+''\,.
\end{align*}
According to equation~(5.3) the curves $X_i$ are given by the 
equations $X_i=0$. Recall that the arguments of $f_+$ and all its derivatives are $x_1+x_2$,
the arguments of $f_-$ and all its derivatives are $x_1-x_2$.

\begin{Le}
\label{180501}
All three curves $X_0,$ $X_1,$ and $X_\infty$ pass through the points $U^l$ and $U^r$\!. 
If $\eps$ is sufficiently close to $\eps_0,$ then the mentioned curves are situated between 
these two points in the following order\textup: $X_\infty$ lies below the curve $x_2=\eps$ 
and above $X_1,$ the curve $X_0$ lies below $X_1$.
\end{Le}

\begin{proof}
The first assertion of the lemma follows from the simple equalities:
\begin{align*}
X_0(x_1,\eps)&=-2\eps D_+(x_1,\eps)\,;
\\
X_1(x_1,\eps)&=2\eps D_+(x_1,\eps)\,;
\\
X_\infty(x_1,\eps)&=2\eps^2 D_+(x_1,\eps)\,.
\end{align*}

To prove the second assertion of the lemma about mutual position of the curves $X_i$ we 
calculate the slopes of these curves at the points $U^l$ and $U^r$. Calculating the 
slope$(X_i)=-\frac{\partial X_i}{\partial x_1}/\frac{\partial X_i}{\partial x_2}$ 
we use the notation from~(5.18):
\eq{180502}{
\kappa_+=\frac{2f_+'''(u_++\eps)}{f_+''(u_++\eps)-f_-''(u_+-\eps)-2\eps f_+'''(u_++\eps)}\,.
}
Since we have here two roots $u^l$ and $u^r$ of the equation $D_+(u,\eps)=0$ instead of 
one root $u_+$ in~(5.18), the corresponding values of $\kappa_+$ will be
denoted $\kappa^l$ and $\kappa^r$. Note that according to~\eqref{160501}--\eqref{160503} 
we have $\kappa^l<0$ and $\kappa^r>0$. Moreover, $\kappa^l\to-\infty$ and $\kappa^r\to+\infty$ 
as $\eps\to\eps_0$ from above. Indeed, in this situation $u^{l,r}\to u_0$, and therefore, denominator 
in~\eqref{180502} tends to 0.

By direct calculations we get (recall that $u_+$ is any of roots $u^{l,r}$, and then 
naturally $\kappa_+$ is the corresponding $\kappa^{l,r}$):
$$
\begin{aligned}
\text{slope}(X_0)(u_+,\eps)&=\frac{f_+''(u_+\!+\eps)-f_-''(u_+\!-\eps)-2\eps f_+'''(u_+\!+\eps)}
{2\eps f_+'''(u_+\!+\eps)}=\frac1{\eps\kappa_+}\,;
\\
\text{slope}(X_1)(u_+,\eps)&=\frac{f_+''(u_+\!+\eps)-f_-''(u_+\!-\eps)-2\eps f_+'''(u_+\!+\eps)}
{f_+''(u_+\!+\eps)-f_-''(u_+\!-\eps)+2\eps f_+'''(u_+\!+\eps)}=\frac1{1+2\eps\kappa_+}\,;
\rule{0pt}{20pt}
\\
\text{slope}(X_\infty)(u_+,\eps)&=\frac{f_+''(u_+\!+\eps)-f_-''(u_+\!-\eps)-
2\eps f_+'''(u_+\!+\eps)}{2f_+''(u_+\!+\eps)-2f_-''(u_+\!-\eps)+2\eps f_+'''(u_+\!+\eps)}
=\frac1{2+3\eps\kappa_+}\,.\rule{0pt}{25pt}
\end{aligned}
$$
Since $\kappa^r>0$, we have $2+3\eps\kappa^r>1+2\eps\kappa^r>\eps\kappa^r>0$, i.\,e., 
at the point $U^r$ we have
$$
0<\text{slope}(X_\infty)(u^r,\eps)<\text{slope}(X_1)(u^r,\eps)<\text{slope}(X_0)(u^r,\eps).
$$
Note that all these curves on the interval between $U^l$ and $U^r$ cannot cross each 
other and are uniformly close to the horizontal line $x_2=\eps$ as $\eps\to\eps_0$. 
The implicit function theorem and condition~\eqref{120301} supply us with the
possibility to speak about mutual positions of these curves for $\eps$ being
slightly bigger than $\eps_0$.
The above chain of the inequalities is sufficient for the fact that the mutual 
position of the considered curves is such as described in the lemma. 
\end{proof}

Mutual position of the curves $X_0$, $X_1$, and $X_\infty$ is illustrated on Fig.~\ref{081202}.

\begin{figure}[h]
    \centering
    \includegraphics[scale = 0.4]{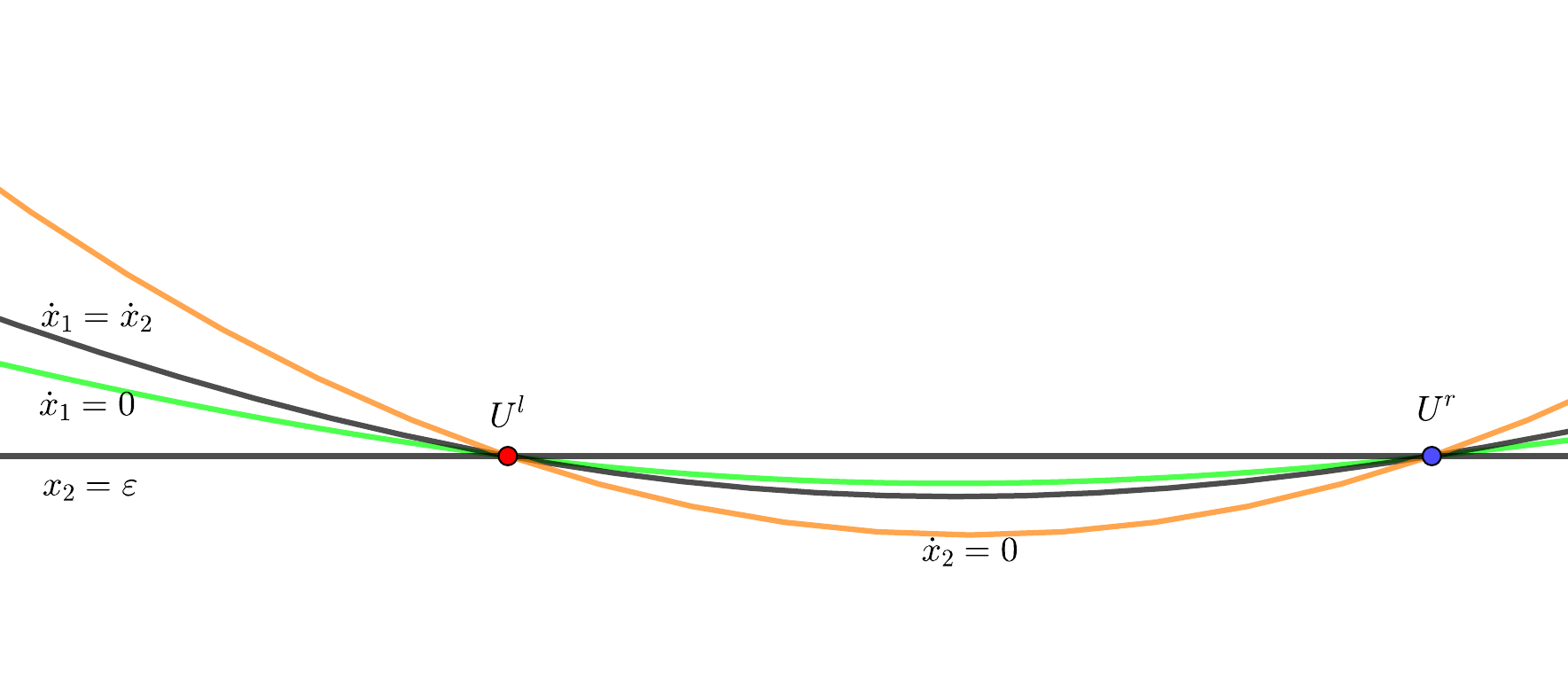}
    \caption{The curves $X_0$, $X_1$, and $X_\infty$ described in Lemma~\ref{180501}}
    \label{081202}
\end{figure}

Now let us consider the field~(5.3) near the stationary points $U^l$ and $U^r$. 
The Jacobian matrix at the points $U^l$ and $U^r$ was calculated in formula~(5.17). 
Let us investigate this matrix in more details for $\eps$ slightly bigger than $\eps_0$. 
It will be a bit more convenient to use a parameter $s=1+\eps\kappa_+$ rather than 
$\kappa_+$. Then the matrix we would like to investigate takes the form
\eq{190501}{
\begin{pmatrix}
1&1-3s
\\
-1&s-1
\end{pmatrix},
}
whose characteristic polynomial is $\lambda^2-s\lambda-2s$ and the eigenvalues are
\eq{190502}{
\lambda=\frac{s\pm\sqrt{s^2+8s}}2\,.
}
The eigenvectors of the matrix~\eqref{190501} are
\eq{190503}{
\begin{pmatrix}
s-1-\lambda
\\
1
\end{pmatrix}.
}
If $s>0$, then we have real eigenvalues of different sign, therefore the stationary point is 
a saddle point. If $s\in(-8,0)$ then we have a pair of complex conjugate eigenvalues, 
and therefore the stationary point is a spiral point. If $s<-8$ then we have a pair of 
negative eigenvalues, and therefore the stationary point is a node. Since $\kappa^r>0$ 
{\bf(}see~\eqref{160502}{\bf)}, our field always has a saddle point at $U^r$. At the 
point $U^l$ all three situations can occur, but we are interested now what happens 
for $\eps$ slightly bigger than $\eps_0$ and in this case a node occurs at $U^l$. 
Indeed, we already know that $\kappa^l\to-\infty$ as $\eps\to\eps_0$ from above, i.\,e., $s\to-\infty$.
Therefore, for $\eps$ slightly bigger than $\eps_0$, we have two negative eigenvalues 
of the matrix~\eqref{190501}.

Now we consider the behavior of eigenvectors~\eqref{190503}. Since
\eq{141004}{
s-1-\lambda=\frac{s-2\mp\sqrt{s^2+8s}}2=\frac{2(1-3s)}{s-2\pm\sqrt{s^2+8s}}\,,
}
i.\,e., for one sign in front of the square root we get $s-1-\lambda=s+1+O(s^{-1})$ for 
the other sign we get $s-1-\lambda\to-3$ as $s\to+\infty$. This means that
one eigenvector at $U^r$ has a small positive slope tending to $0$, the slope of the second 
eigenvector tends to $-1/3$ as $\eps\to\eps_0$ from above, and we have the behavior of the integral curves 
of our field as on Fig.~\ref{210501}. 

\begin{figure}[h]
    \centering
    \includegraphics[scale = 0.7]{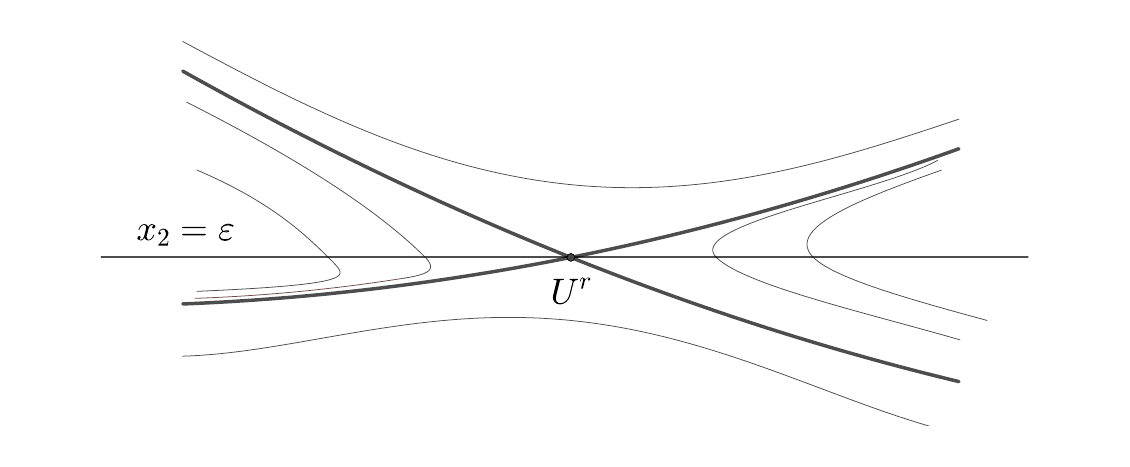}
    \caption{The integral curves near the point $U^r$}
    \label{210501}
\end{figure}

One eigenvector at $U^l$ has a small 
negative slope tending to $0$ as well, the slope of the second eigenvector also tends to $-1/3$, 
and we have the behavior of the integral curves of our field as on Fig.~\ref{210502}.

\begin{figure}[h]
    \centering
    \includegraphics[scale = 0.6]{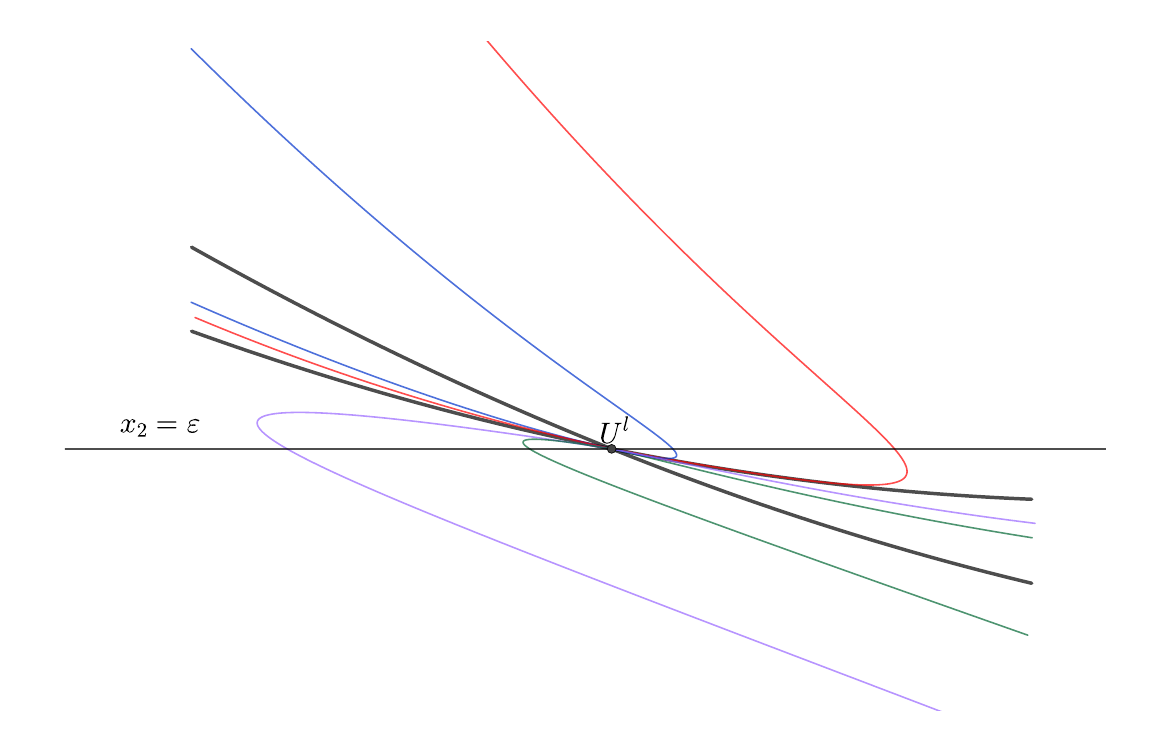}
    \caption{{The integral curves near the point $U^l$}}
    \label{210502}
\end{figure}

Consider the integral curve (let us call it $\ell_\eps$) starting at $U^r$ along the eigenvector of 
the Jacobian matrix that has a small positive slope. After shifting right by $\eps$, 
this curve will be the spine of our left herringbone. To this aim
we need to verify that $\ell_\eps$ runs inside the domain
where $D_+>0$ and $D_->0$.

First we note that $\ell_\eps$ can leave the strip across the upper boundary somewhere on the interval $[U^l,U^r)$, 
because from the lower boundary and from the part of the upper boundary which is to the left from $U^l$ it is separated
by the integral curve that pass through the point $U^l$
with a slope close to $-\tfrac13$. As it was already mentioned,
$s-1-\lambda\approx s+1=2+\eps\kappa^r$, that is, $\ell_\eps$
starts at $U^r$ between the curves $X_0$ and $X_1$, since $\eps\kappa^r<2+\eps\kappa^r<1+2\eps\kappa^r$ (if $\eps$ is not too far from $\eps_0$, i.\,e., if $\eps\kappa^r$ is sufficiently 
large). Since the integral lines of the field have slope equal
to $1$ at any point of $X_1$ and the curve $X_1$ itself has
the slope close to zero for $x_1\in(u^l,u^r)$, the integral
lines intersect $X_1$ from the top right to the bottom left.
Therefore, our curve $\ell_\eps$ cannot intersect $X_1$ inside
the interval $(u^l,u^r)$, because near the point $U^r$ it goes
under the curve $X_1$. Thus, $\ell_\eps$ leaves the strip $\Oe$
through the point $U^l$, being inside the strip below the line
$X_1$, i.\,e., in the domain $D_+>0$.

Verification that the slope of the spine is grater than $-1$ 
is reduced to the verification of the inequality 
$D_-(u_0,\eps_0)>0$,
because then for $\eps$ sufficiently close to $\eps_0$
the inequality $D_-(x)$ would be fulfilled on the part of 
$\ell_\eps$ inside the strip.
The inequalities~\eqref{150506} and~\eqref{120301} imply:
$$
f_+''(u_0+\eps_0)-f_-''(u_0-\eps_0)>0\,,
$$
whence
$$
D_-(u_0,\eps_0)=D_+(u_0,\eps_0)+
f_+''(u_0+\eps_0)-f_-''(u_0-\eps_0)>0\,.
$$

The above consideration almost proves Theorem~\ref{160504}. 
Recall that the spine of left herringbone is shifted to the 
right by $\eps$ in comparison with the integral curve 
$\ell_\eps$. Thus, we have constructed the domain 
$\Om{W+}(u^l,u^r)$ with a left 
herringbone foliated this domain. 

It remains to check that from both sides of this domain 
we are able to construct simple right foliation, i.\,e., to 
check that conditions~(2.2) are fulfilled for $u\in(u_1,u^l)$ 
and for $u\in(u^r,u_2)$ for some $u_1$ and $u_2$. 
However, this is a direct consequence of the definition 
of the points $u^r$ and $u^l$: they are the neighbor simple 
roots of the equation $D_+(u,\eps)=0$ such that $D_+(u,\eps)<0$ 
for $u^l<u<u^r$. This means that $D_+(u,\eps)>0$ on some 
intervals $(u_1,u^l)$ and $(u^r,u_2)$. Since we have assumed 
that conditions~\eqref{150501} and~\eqref{150502} do not 
degenerate simultaneously at the point $\eps=\eps_0$, i.\,e., 
the condition $D_-(u,\eps)>0$ remains true on $[u_1,u_2]$, we 
can construct the simple right foliation on $\Om{R}(u_1,u^l)$
and on $\Om{R}(u^r,u_2)$.
\end{proof}

We will often use the symbol $\ell_\eps$, so let us define it 
rigorously. 

\begin{Def}
Symbol $\ell_\eps$ denotes the part of the integral curve of 
the field~(5.3) that goes inside the strip $\Oe$ and passes 
with a positive slope through the stationary point $(u,\eps)$, 
where $D_+(u,\eps)=0$.
\end{Def}

\begin{Rem}
The same symbol $\ell_\eps$ will denote analogous curves in
symmetrical cases. Would be necessary to consider two such
curves simulteneously, we add some index, for example,
$\ell_\eps^+$ и $\ell_\eps^-$. 
\end{Rem}

Now we would like to investigate dependence on $\eps$ of 
the mutual position of the curves $\ell_\eps$.
In fact, we repeat the reasoning contained in the proof of 
Proposition~6.1.

\begin{Prop}
\label{141003}
Let $D_+$ vanish at the point $(u,\eps)$. Assume that 
\eq{070304}{
\kappa_+^{\LL}\df\frac{2f'''_+(u+\eps)}
{f''_+(u+\eps)-f''_-(u-\eps)-2\eps f'''_+(u+\eps)}>0\,. 
}
Then\textup, while the curve $\ell_\eps$ goes inside the domain 
where $x_2D_+D_->0$, its slope is strictly between $-1$ and $1$ 
\textup(i.\,e.\textup, it can be considered as a graph of a 
function of the argument $x_1$\textup). Moreover\textup, for 
$\eps'<\eps''$ the curve $\ell_{\eps'}$ goes strictly above 
the curve $\ell_{\eps''}$ while they are inside this domain.
\end{Prop}

\begin{proof}
Since we consider the strict inequality $x_2D_+D_->0$ and our 
curve $\ell_\eps$ starts from the upper boundary of the strip, 
we consider this curve only while it is inside the upper half 
of the strip. Moreover, since the beginning part of the curve 
$\ell_\eps$ goes in the domain where $D_+>0$ and $D_->0$, we 
consider this curve only while these conditions are fulfilled 
as well. This implies that the slope of $\ell_\eps$ is strictly 
between $-1$ and $1$. Therefore, the assertion that one such 
curve is below than other is meaningful.

Define the function $\eps\mapsto u(\eps)$ implicitly by the 
equation $D_+(u,\eps)=0$, i.\,e.,
$$
f_+'(u+\eps)-f_-'(u-\eps)-2\eps f_+''(u+\eps)=0.
$$
Taking the derivative with respect to $\eps$ we get
$$
u'(\eps)=\frac{f_+''-f_-''+2\eps f_+'''}
{f_+''-f_-''-2\eps f_+'''}=1+2\eps\kappa_+^\LL.
$$ 
Therefore, the slope of this curve is
$$
\frac1{u'}=\frac1{1+2\eps\kappa_+^\LL}=\frac1{2s-1},
$$
where $s=1+\eps\kappa_+^\LL$. Note that according the 
hypotheses of the proposition we have 
$\kappa_+^\LL>0$, and therefore $s>1$.

On the other side, the curve $\ell_\eps$ starts at the point 
$(u,\eps)$ along the eigenvector of the Jacobian matrix with 
the slope
$$
\frac1{s-1-\lambda}=\frac2{s-2+\sqrt{s^2+8s}}, 
$$
{\bf(}see~\eqref{141004}{\bf)}.

It is easy to check that
$$
2s-1>\frac{s-2+\sqrt{s^2+8s}}2
$$
for $s>1$, i.\,e., the slope of $\ell_\eps$ is strictly 
bigger than the slope of the curve $u(\eps)$ at their common 
point $(u(\eps),\eps)$. Thus we can conclude that for $\eps'$ 
sufficiently close to $\eps''$ the point $(u(\eps'),\eps')$ 
lies above the curve $\ell_{\eps''}$. Therefore, the whole 
curve $\ell_{\eps'}$ is above $\ell_{\eps''}$ in some 
neighborhood of this point. We continue these curves to the 
left and assume that they have a crossing point. 
Since the curve $\ell_{\eps'}$ was above the curve 
$\ell_{\eps''}$ till this point, its slope has to be greater 
than the slope of the second curve. However, according to 
Proposition~5.4, this is impossible inside the domain where 
$x_2D_+D_->0$.
\end{proof}

In order not to be restricted to words that in symmetric cases 
everything is similar, we formulate these symmetric statements.

\begin{Prop}
\label{070301}
Let $D_-$ vanish at the point $(u,-\eps)$. Assume that 
\eq{070305}{
\kappa_-^\LL\df\frac{2f'''_-(u+\eps)}
{f''_-(u+\eps)-f''_+(u-\eps)-2\eps f'''_-(u+\eps)}>0\,. 
}
Denote by $\ell_\eps$ the integral curve of the field 
\textup{(5.3)} inside the strip $\Oe$ with a negative slope 
at the point $(u,-\eps)$. Then\textup, while the curve 
$\ell_\eps$ goes inside the domain where $x_2D_+D_-<0$, 
its slope is strictly between $-1$ and $1$ 
\textup(i.\,e.\textup, it can be considered as a graph of 
a function of the argument $x_1$\textup). Moreover\textup, 
for $\eps'<\eps''$ the curve $\ell_{\eps'}$ goes strictly 
below the curve $\ell_{\eps''}$ while they are inside 
this domain.
\end{Prop}

Last two statements are related to the construction of left 
horizontal herringbones with ends on the upper and lower 
boundaries of the strip, respectively. However, if we are to 
deal with right herringbones, we will need statements related 
to the situation obtained by reflection with respect to the 
vertical axis. The spines of the right herringbones we 
obtain if we shift left by $\eps$ the integral curve 
of the following field:

\eq{070302}{ \begin{cases} \dot x_1=
(\eps+x_2)\big[\half(f_-'-f_+')-x_2f_+''\big] +
(\eps-x_2)\big[\half(f_-'-f_+')-x_2f_-''\big]\,;
\\
\dot x_2=(\eps+x_2)\big[\half(f_-'-f_+')-x_2f_+''\big] -
(\eps-x_2)\big[\half(f_-'-f_+')-x_2f_-''\big]\,, 
\end{cases} 
}
or in a slightly different form:
\eq{070303}{ \begin{cases} \dot x_1=
\eps(f_-'-f_+')-x_2(\eps+x_2)f_+''-x_2(\eps-x_2)f_-''\,; 
\\
\dot x_2=x_2(f_-'-f_+')-x_2(\eps+x_2)f_+''+
x_2(\eps-x_2)f_-''\,, 
\end{cases} 
}
where $f_+=f_+(x_1-x_2)$ and $f_-=f_-(x_1+x_2)$.

\begin{Prop}
\label{080301}
Let $D_-$ vanish at the point $(u,\eps)$. Assume that 
\eq{070306}{
\kappa_-^\RR\df\frac{2f'''_-(u-\eps)}
{f''_+(u+\eps)-f''_-(u-\eps)-2\eps f'''_-(u-\eps)}>0\,. 
}
Denote by $\ell_\eps$ the integral curve of the 
field~\eqref{070302} inside the strip $\Oe$ with a positive 
slope at the point $(u,-\eps)$. Then\textup, while the 
conjugate of the curve $\ell_\eps$ goes inside the domain 
where $x_2D_+D_->0$, its slope is strictly between $-1$ and $1$ 
\textup(i.\,e.\textup, it can be considered as a graph of 
a function of the argument $x_1$\textup). Moreover\textup, 
for $\eps'<\eps''$ the curve $\ell_{\eps'}$ goes strictly 
below the curve $\ell_{\eps''}$ while the conjugates of 
these curves are inside the domain where $x_2D_+D_->0$. 
\end{Prop}

\begin{Prop}
\label{080302}
Let $D_+$ vanish at the point $(u,-\eps)$. Assume that 
\eq{070307}{
\kappa_+^\RR\df\frac{2f'''_+(u-\eps)}
{f''_-(u+\eps)-f''_+(u-\eps)-2\eps f'''_+(u-\eps)}>0\,. 
}
Denote by $\ell_\eps$ the integral curve of the 
field~\eqref{070302} inside the strip $\Oe$ with a negative 
slope at the point $(u,\eps)$. Then\textup, while the 
conjugate of the curve $\ell_\eps$ goes inside the domain 
where $x_2D_+D_-<0$, its slope is strictly between $-1$ and $1$ 
\textup(i.\,e.\textup, it can be considered as a graph of 
a function of the argument $x_1$\textup). Moreover\textup, 
for $\eps'<\eps''$ the curve $\ell_{\eps'}$ goes strictly 
above the curve $\ell_{\eps''}$ while the conjugates of 
these curves are inside the domain where $x_2D_+D_-<0$.
\end{Prop}

Let us make a remark that will not be needed just now, however 
in the future it allow us not to repeat the reasoning already 
given, and it will be sufficient to refer to this proposition.

\begin{Rem}
In the proposition just proved, it was assumed that we 
consider the curve $\ell_\eps$ only in the upper or lower 
half of the strip. However, it is not difficult to extend 
this assertion to the other half if the integral curve 
$\ell_\eps$ crosses the midline of the strip at the point 
$(u_0,0)$, where $u_0$ is a simple root of the equation 
$f'_+(u)=f'_-(u)$. In this case when passing to the other 
half of the strip, the relations over or under are reversed.
The general rule can be formulated as follows: the larger 
$\eps$, the closer the curve $\ell_\eps$ to the midline of 
the strip $x_2=0$.
\end{Rem}

\section{Examples. Polynomials of third degree: \\
the case of one minor pocket}

In Section~7, were we started consideration of the third 
degree polynomials
$$
f_\pm(t)=a_3^\pm t^3+a_2^\pm t^2+a_1^\pm t+a_0^\pm\,,
$$
the case $a_3^+=a_3^-$ was described in full generality, and 
the foliation in the case $a_3^+\ne a_3^-$ was investigated 
for sufficiently small $\eps$ only. Now we are ready 
to describe the foliation in the case when the discriminant 
of the quadratic polynomial 
$$
f_+'(t)-f_-'(t)=3(a_3^+-a_3^-)t^2+2(a_2^+-a_2^-)t+(a_1^+-a_1^-)
$$
is negative, i.\,e.,
\eq{141001}{
3(a_3^+-a_3^-)(a_1^+-a_1^-)>(a_2^+-a_2^-)^2.
}

Using the symmetry we can assume  $|a_3^+|>|a_3^-|$ 
(if not we interchange $f_+$ and $f_-$) and $a_3^+>0$ 
(if not we replace $t$ by $-t$). We will investigate 
now the field~(5.3) to show that the situation falls 
under the scope of Theorem~\ref{160504} for
\eq{190302}{
\eps_0=\eps_0^+\df\sqrt{\frac1{12a_3^+}
\Big[a_1^+-a_1^--\frac{(a_2^+-a_2^-)^2}{3(a_3^+-a_3^-)}\Big]}\,.
}

We start with calculation of the functions $D_\pm$:
\eq{130301}{
2x_2D_+(x)=3(a_3^+-a_3^-)(x_1-x_2)^2+2(a_2^+-a_2^-)(x_1-x_2)+
(a_1^+-a_1^-)-12a_3^+x_2^2\,;
}
\eq{130302}{
2x_2D_-(x)=3(a_3^+-a_3^-)(x_1+x_2)^2+2(a_2^+-a_2^-)(x_1+x_2)+
(a_1^+-a_1^-)+12a_3^-x_2^2\,.
}
We see that $D_+(u,\eps)\ge0$ and $D_-(u,\eps)\ge0$ for all 
$u$ if $\eps\le\eps_0$ and, therefore, we have the right 
simple foliation $\Omega_\eps=\Om{R}(-\infty,+\infty)$. 
For $\eps=\eps_0^+$ and
\eq{140301}{
u_0^+=\eps_0^+-\frac{a_2^+-a_2^-}{3(a_3^+-a_3^-)}
}
condition~\eqref{150503} is fulfilled. Moreover, $\eps_0^+$ 
is a simple root, i.\,e., all 
conditions~\eqref{150505}--\eqref{120301} are fulfilled 
as well. It will be evident if we rewrite the expression 
from~\eqref{130301} in the form:
\eq{190301}{
2x_2D_+(x)=3(a_3^+-a_3^-)\big((x_1-u_0^+)-(x_2-\eps_0^+)\big)^2
-12a_3^+\big(x_2^2-(\eps_0^+)^2\big)\,.
}

Thus, from Theorem~\ref{160504} we know that around the point 
$(u_0^++\eps_0^+,\eps_0^+)$ a minor pocket begins to grow. 
The question is what happens if the parameter $\eps$ is 
not very close to $\eps_0^+$. We will see that the situation 
depends on the sign of $a_3^-$.

First, we consider the case $a_3^-\ge0$. When 
$\eps\le\eps_0^+$ we have the simple right foliation:
\eq{160402}{
\Omega_\eps=\Om{R}(-\infty,+\infty)\,,
}
because for such $\eps$ we have $D_+\ge0$ and $D_-\ge0$ on 
the upper boundary (in fact, in the whole strip) and we apply 
Proposition~2.2 (\cite{First}). We would like to show that 
for all $\eps$ greater than $\eps_0^+$ the foliation is 
described by the decomposition
\eq{140302}{
\Omega_\eps=\Om{R}(-\infty,v_+)\cup
\Om{W$+$}(v_+,u_+^r)\cup\Om{R}(u_+^r,+\infty)\,,
}
i.\,e., a herringbone is extended from the point 
$(v_++\eps,\,\eps)$ till the point $(u_+^r+\eps,\,\eps)$.

Looking on the expressions in~\eqref{190301} we see that 
the equation $D_+(u,\eps)=0$ has exactly two roots $u_+^l$ 
and $u_+^r$ if $\eps>\eps_0^+$:
\eq{141002}{
u_+^{l,r}=u_0^++(\eps-\eps_0^+)\pm2
\sqrt{\frac{a_3^+}{a_3^+-a_3^-}\big(\eps^2-(\eps_0^+)^2\big)}.
}
From~\eqref{130302} we see that $D_-(u,\eps)>0$ for all 
$u$ and $\eps$. Therefore, to construct the foliation 
presented in~\eqref{140302} we need to find an integral 
curve of the field~(5.3) connecting the points $(v_+,\eps)$ 
and $(u_+^r,\eps)$.

The Jacobian matrix at the points $U_+^r$ and $U_+^l$ was 
partially investigated after the proof of Lemma~\ref{180501}, 
where the matrix proportional to the Jacobian was written 
in the form {\bf(}see~\eqref{190501}{\bf)}:
\eq{150301}{
\begin{pmatrix}
1&1-3s
\\
-1&s-1
\end{pmatrix},
}
where $s=1+\eps\kappa_+$ and $\kappa_+$ is given 
by~\eqref{180502}. In the example under consideration we have 
\eq{150302}{
\kappa_+^l=-\kappa_+^r\qquad\text{and}\qquad\kappa_+^r
=\sqrt{\frac{a_3^+}{(a_3^+-a_3^-)(\eps^2-(\eps_0^+)^2)}}\,.
}

The characteristic polynomial of the matrix~\eqref{150301} 
is $\lambda^2-s\lambda-2s$ and the eigenvalues are
\eq{150303}{
\lambda=\frac{s\pm\sqrt{s^2+8s}}2\,.
}
The corresponding eigenvectors are
\eq{150304}{
\begin{pmatrix}
s-1-\lambda
\\
1
\end{pmatrix}.
}

At $U_+^r$ we always have a saddle point and the picture of 
integral curves near this point looks like it is shown at 
Pic.~\ref{210501}. The integral curve we are interested in 
starts from $U_+^r$ along the eigenvector corresponding to 
$\lambda=\half(s-\sqrt{s^2+8s})$, 
i.\,e., its slope at $U_+^r$ is
$$
\frac1{s-1-\lambda}=\frac2{s-2+\sqrt{s^2+8s}}\,.
$$
When $\eps$ grows from $\eps_0^+$ till $+\infty$, 
the parameter $s$ decreases from $+\infty$ till some value 
greater than $1$, and therefore, this slope monotonically 
increases from $0$ till some value less than $1$.

As it was mentioned, the type of the stationary point 
$U_+^l$ depends on~$\eps$. We know that for $\eps$ sufficiently 
close to $\eps_0^+$ (when $s<-8$) the Jacobian matrix 
at $U_+^l$ has two negative eigenvalues, therefore $U_+^l$ 
is a node, and the picture of the integral curves near this 
point looks like at Pic.~\ref{210502}. In such a situation 
the integral curve we are interested in intersects the boundary 
exactly at the point $U_+^l$. This situation occurs for all 
$\eps$ if $80a_3^+\le81a_3^-$. Indeed, then 
$81(a_3^+-a_3^-)\le a_3^+$, whence
$$
s=1+\eps\kappa_+^l=1-
\sqrt{\frac{a_3^+\eps^2}{(a_3^+-a_3^-)(\eps^2-(\eps_0^+)^2)}}
\le 1-\sqrt{\frac{81\eps^2}{\eps^2-(\eps_0^+)^2}}<1-9=-8\,.
$$

If $81(a_3^+-a_3^-)>a_3^+$, then for
\eq{160401}{
\eps<\eps_1^+\df9\eps_0^+
\sqrt{\frac{a_3^+-a_3^-}{80a_3^+-81a_3^-}}
}
we have $s<-8$ and the point $U_+^l$ is a node and our 
integral curve $\ell_\eps$ intersects the upper boundary 
at the point $U_+^l$. For $\eps=\eps_1^+$ the eigenvalues 
of the Jacobian matrix are equal, and $J(U_+^l)$ is 
a Jordan cell, i.\,e., $U_+^l$ is an improper node and 
our integral curve still passes through $U^l_+$. 
The picture of integral curves is similar to the one 
shown at Pic.~\ref{210502}. The narrow angle between
the integral curves going along the eigenvectors of the 
Jacobian matrix (thicker curves at Pic.~\ref{210502}) 
shrinks as $\eps\nearrow\eps_1^+$ and for $\eps=\eps_1^+$ 
they turn into one curve.

The picture changes for $\eps>\eps_1^+$, when $s\in(-8,0)$ 
and $U_+^l$ is a spiral point. Then our integral curve 
cannot cross the upper boundary not only on the interval 
$(u_+^l,u_+^r)$, but at the point $U_+^l$ as well, i.\,e., 
for $\eps>\eps_1^+$ we have $v_+<u_+^l$. Thus, we have got 
what we wanted: for $a_3^-\ge0$ we have found an integral 
curve of our field that connects two points of the upper 
boundary with first coordinates $v_+$ ($v_+\le u_+^l$) and 
$u_+^r$, i.\,e., we can construct foliation~\eqref{140302}. 
Indeed, the condition $D_+>0$ is fulfilled on $\ell_\eps$ 
because the curve $\ell_\eps$ goes strictly below the curve 
$X_1$, where $D_+=0$, and $D_+>0$ on the whole upper 
boundary except the interval $[u_+^l,u_+^r]$. The inequality 
$D_->0$ is true everywhere in the strip $\Oe$.

Let us say a few words about the behavior of the end points 
of our pocket as $\eps\to\infty$. First, it should be mentioned 
that $v_+$ always exists. Indeed, the curve $\ell_\eps$ can 
be considered as a graph of a function of $x_1$ 
(see Proposition~\ref{141003}) and it cannot intersect the
midline of the strip. Therefore, it is sufficient to ascertain
that the second coordinate of the point on the curve cannot be 
bounded as $x_1\to-\infty$. However, this is true for any
integral curve, because if $x_2$ is fixed and $x_1$ has big 
negative value then both $D_+$ and $D_-$ are approximately
$3(a_3^+-a_3^-)x_1^2/2x_2$, and therefore, the slope of an 
integral curve tends to $-x_2/\eps$, i.\,e., integral curves 
have negative slope and are convex. Hence, if we go along 
such a curve to the left we must intersect the upper bound 
of the strip.

Secondly, from the expressions in~\eqref{141002} we can 
conclude that the right end point of the pocket is
$$
u_+^r+\eps=2\eps\Bigg(1+\sqrt{\frac{a_3^+}{a_3^+-a_3^-}}\Bigg)
+u_0^+-\eps_0^++o(1)\,,
$$
which tends always to $+\infty$. The left end point of 
the pocket is $v_++\eps$, which tends to~$-\infty$ 
if $a_3^->0$, because
$$
u_+^l+\eps=2\eps\Bigg(1-\sqrt{\frac{a_3^+}{a_3^+-a_3^-}}\Bigg)
+u_0^+-\eps_0^++o(1)\,.
$$
If $a_3^-=0$ we cannot say anything about the asymptotic 
behavior of the left end point, because $u_+^l+\eps$ is 
bounded in this case, but we have no information about $v_+$, 
except the inequality $v_+<u_+^l$ for large $\eps$.

\section{Examples. Polynomials of third degree: 
\\the case of two minor pocket}

Now, we turn to a more complicated case $a_3^-<0$. Recall 
that we have assumed that $a_3^+\ge|a_3^-|$, i.\,e., we 
have to consider $a_3^-\in[-a_3^+,0)$. In this situation 
we have one more critical value of $\eps$:
\eq{190303}{
\eps_0^-=\sqrt{\frac1{12|a_3^-|}\Big[a_1^+-a_1^--
\frac{(a_2^+-a_2^-)^2}{3(a_3^+-a_3^-)}\Big]}=
\sqrt{\frac{a_3^+}{|a_3^-|}}\eps_0^+\,.
}
This is the value of $\eps$, for which a root of $D_-$ on 
the boundary of the strip appears:
$$
D_-(u_0^-,\eps_0^-)=0
$$
namely,
\eq{140303}{
u_0^-=-\eps_0^--\frac{a_2^+-a_2^-}{3(a_3^+-a_3^-)}\,.
}

For $\eps\le\eps_0^+$ we have the simple right foliation 
as before:
\eq{160405}{
\Omega_\eps=\Om{R}(-\infty,+\infty)\,.
}
At the moment $\eps=\eps_0^+$ a minor pocket on the upper 
boundary appears and the foliation becomes
\eq{160403}{
\Omega_\eps=\Om{R}(-\infty,v_+)\cup\Om{E$+$}(v_+,u_+^r)\cup\Om{R}(u_+^r,+\infty)\,,
}
where $v_+=u_+^l$ if $\eps_0^+\le\eps\le\eps_1^+$ 
{\bf(}see~\eqref{160401}{\bf)}. If $\eps>\eps_1^+$ 
we can say only that $v_+<u_+^l$. To calculate the value 
of $v_+(\eps)$ we need to solve the Cauchy problem~(5.3) 
with the initial data $x_1(0)=u_+^r$ and $x_2(0)=\eps$. 
Since $(u_+^r,\eps)$ is a saddle point of the field, we 
need to specify the initial slope of the integral curve. 
It has to be the slope of one of eigenvectors of the 
Jacobian matrix. As it was told, we are interested in 
the integral curve $\ell_\eps$ starting along the 
eigenvector~\eqref{190503} with positive slope, i.\,e., 
with $\lambda=\half\big(s-\sqrt{s^2+8s}\big)$ and 
$s=1+\eps\kappa_+^r$. The value of $v_+(\eps)$ is the first 
coordinate of the next crossing point of this integral curve 
with the boundary $x_2=\eps$.

If 
\eq{301101}{
v_+(\eps_0^-)>u_0^-,
}
then at the moment $\eps=\eps_0^-$ a minor pocket on the 
lower boundary appears. We can guarantee 
inequality~\eqref{301101} only in the case when we have 
an explicit formula for $v_+$. This occurs in the case 
$80(a_3^+)^2\le81(a_3^-)^2$. Indeed, in such a case we have
$$
\frac{a_3^+}{|a_3^-|}\le
\frac{81(a_3^++|a_3^-|)}{80a_3^++81|a_3^-|}\,,
$$
or
$$
\eps_0^+\sqrt\frac{a_3^+}{|a_3^-|}\le
9\eps_0^+\sqrt\frac{a_3^+-a_3^-}{80a_3^+-81a_3^-}\,,
$$
where the expression on the left hand side is $\eps_0^-$ 
{\bf(}see~\eqref{190303}{\bf)}, whereas the expression on 
the right hand side is $\eps_1^+$ 
{\bf(}see~\eqref{160401}{\bf)}. Therefore, for 
$80(a_3^+)^2\le81(a_3^-)^2$ we have 
$v_+(\eps_0^-)=u_+^l(\eps_0^-)$. Moreover, we always have 
$u_+^l(\eps_0^-)>u_0^-$. Indeed, according to~\eqref{141002}
$$
u_+^l(\eps_0^-)=u_0^++(\eps_0^--\eps_0^+)
-2\sqrt{\frac{a_3^+}{a_3^+-a_3^-}
\big((\eps_0^-)^2-(\eps_0^+)^2\big)}\;.
$$
If we plug in here expression~\eqref{140301} for $u_0^+$, 
expression~\eqref{190303} for $\eps_0^-$, and use 
expression~\eqref{140303} for $u_0^-$, we get
$$
u_+^l(\eps_0^-)-u_0^-=2\eps_0^+\sqrt\frac{a_3^+}{|a_3^-|}
\bigg(1-\sqrt{\frac{a_3^+-|a_3^-|}{a_3^++|a_3^-|}}\;\bigg)>0\,.
$$

In general, the inequality~\eqref{301101} can fail. 
We consider the foliation in such a case a bit later, 
but from now on we will assume that~\eqref{301101} is 
fulfilled. Then we have foliation~\eqref{160403} for 
$\eps\in(\eps_0^+,\eps_0^-]$, because at the moment 
$\eps=\eps_0^-$ a minor pocket on the lower boundary appears.

For $\eps>\eps_0^-$ we have two roots $u_-^{l,r}$ of the 
equation $D_-(u,\eps)=0$. If we rewrite~\eqref{130302} as
$$
2x_2D_-(x)=3(a_3^+-a_3^-)\big((x_1-u_0^-)+
(x_2+\eps_0^-)\big)^2+12a_3^-(x_2^2-(\eps_0^-)^2)\,,
$$
then it is clear that
\eq{090401}{
u_-^{l,r}=u_0^--\eps+\eps_0^-
\pm2\sqrt{\frac{|a_3^-|}{a_3^+-a_3^-}
\big(\eps^2-(\eps_0^-)^2\big)}\,.
}

Thus, we are under the scope of Theorem~\ref{120302}: 
at the point $(u_0^--\eps_0^-,-\eps_0^-)$  a minor pocket 
starts to grow on the lower boundary. The new pocket is 
a right herringbone, therefore we have to investigate the 
field~\eqref{070302}. Recall that the spine of the right 
herringbone we are interested in is shifted by $\eps$ to 
the left with respect to corresponding integral curve of 
this field.

The Jacobian matrix (compare with~(5.5) of~\cite{First}) 
of this field is
\eq{080403}{
J(x)\!=\!\begin{pmatrix}
\scriptstyle\eps(f_-''-f_+'')-
x_2(\eps-x_2)f_-'''-x_2(\eps+x_2)f_+'''&
\scriptstyle 2x_2(f_-''-f_+'')-
x_2(\eps-x_2)f_-'''+x_2(\eps+x_2)f_+'''
\\
\scriptstyle x_2(f_-''-f_+'')+
x_2(\eps-x_2)f_-'''-x_2(\eps+x_2)f_+'''&
\scriptstyle(f_-'-f_+')+(\eps-x_2)(f_-''+x_2f_-''')-
(\eps+x_2)(f_+''-x_2f_+''')
\end{pmatrix}.
}

For $\eps>\eps_0^-$ we have two stationary points 
of~\eqref{070302} on the lower boundary: 
$U_-^{l,r}=(u_-^{l,r},-\eps)$. At these points
\eq{100401}{
J(U_-)=\eps(f_-''-f_+''+2\eps f_-''')
\begin{pmatrix}
1&1-3s
\\
-1&s-1
\end{pmatrix},
}
where $s=1+\eps\kappa_-$ and 
\eq{100402}{
\kappa_-=\kappa_-^\RR=\frac{2f_-'''(u_--\eps)}
{f_+''(u_-+\eps)-f_-''(u_--\eps)-2\eps f_-'''(u_--\eps)}\,.
}
In the example under consideration we have 
\eq{100403}{
\kappa_-^l=-\kappa_-^r\qquad\text{and}\qquad\kappa_-^l
=\sqrt{\frac{|a_3^-|}{(a_3^+-a_3^-)(\eps^2-(\eps_0^-)^2)}}\,.
}

In the case of the pocket on the upper boundary we have
already considered, the integral curve starts at $U_+^r$. 
We called it $\ell_\eps$, but now we will add an index~$+$ 
to distinguish $\ell_\eps^+$ from the similar integral 
curve $\ell_\eps^-$ started at $U_-^l$. The point $U_-^l$ 
is always a saddle point, and $\ell_\eps^-$ goes along 
the eigenvector of Jacobian matrix~\eqref{080403} 
of the form~\eqref{150304} with 
$\lambda=\half\big(s-\sqrt{s^2+8s}\big)$. Therefore, 
when $\eps$ increases from $\eps_0^-$ till infinity 
the parameter $s$ decreases from $+\infty$ till some 
value greater than $1$, and the slope of $\ell_\eps^-$ 
at $U_-^l$, which is
$$
\frac1{s-1-\lambda}=\frac2{s-2+\sqrt{s^2+8s}}\,,
$$
increases from $0$ till some value less than $1$. The 
curve $\ell_\eps^+$ being shifted right by $\eps$ 
was the spine of a left herringbone and generated a pocket 
on the upper boundary. In a similar way, $\ell_\eps^-$ being 
shifted left by $\eps$ will be the spine of a right 
herringbone and will form a pocket on the lower boundary.

The character of the second stationary point $U_-^r$ depends 
on~$\eps$. We know that for $\eps$ sufficiently close to 
$\eps_0^-$ (when $s<-8$) the Jacobian matrix has two negative 
eigenvalues. This situation occurs for
\eq{200401}{
\eps_0^-<\eps<\eps_1^-\df9\eps_0^-
\sqrt{\frac{a_3^+-a_3^-}{81a_3^+-80a_3^-}}.
}
For such $\eps$ we have $s<-8$, the point $U_-^r$ is 
a node, and $\ell_\eps^-$ intersects the lower boundary 
at the point $U_-^r$. For $\eps=\eps_1^-$ the eigenvalues 
of the Jacobian matrix are equal and $J(U_-^r)$ is 
a Jordan cell, i.\,e., $U_-^r$ is an improper node and 
$\ell_\eps^-$ still passes through $U_-^r$. The picture 
changes for $\eps>\eps_1^-$, when $s\in(-8,0)$ and $U_-^r$ 
is a spiral point. Then our integral curve cannot cross 
the lower boundary not only on the interval $(u_-^l,u_-^r)$, 
but at the point $U_-^r$ as well, i.\,e., for $\eps>\eps_1$ 
we have some point $(v_-,-\eps)$ with $v_-=v_-(\eps)>u_-^r$, 
where our integral curve cross the lower boundary at the 
second time. Thus, we have found two integral curves: an 
integral curve $\ell_\eps^+$ of the field~(5.3) that connects 
two points of the upper boundary with the first coordinates 
$v_+$ ($v_+\le u_+^l$) and $u_+^r$, and another integral 
curve $\ell_\eps^-$ of the field~\eqref{070302} that connects 
two points of the lower boundary with the first coordinates 
$u_-^l$ and $v_-$ ($v_-\ge u_-^r$). 

Recall that we are currently assuming that at the moment 
a pocket appears on the lower boundary, the left end of 
the pocket on the upper boundary has not moved too far 
to the left, specifically, inequality~\eqref{301101} holds. 
Then, at the moment $\eps_0^-$, a pocket forms on the lower boundary with endpoints $u_-^l-\eps$ and $v_--\eps$. We know that as $\eps$ increases, the curve $\ell_\eps^+$ descends, and the value $v_+$ decreases. Similarly, the curve $\ell_\eps^-$ rises, and $v_-$ increases. Therefore, a moment will come, which we denote as $\eps_2$, when $v_-=v_+$. Then, for $\eps\in(\eps_0^-,\eps_2]$, we have the foliation
\eq{160404}{
\Omega_\eps=\Om{R}(-\infty,u_-^l)\cup\Om{E$-$}(u_-^l,v_-)\cup\Om{R}(v_-,v_+)
\cup\Om{W$+$}(v_+,u_+^r)\cup\Om{R}(u_+^r,+\infty)\,.
}

At the moment $\eps=\eps_2$, the region separating the pockets with the simple right foliation $\Om{R}(v_-,v_+)$ collapses into a single line. To describe what happens next, we will need a new element in our collection of elementary foliations, which we will introduce in the following section.

\section{Rectangle}
\label{211102}
We will need a new notation for the areas foliated by the 
herringbones. Previously, when all herringbones started and 
ended on the boundary, the corresponding area could 
conveniently be measured by the segment cut out by the 
herringbone from the midline of the strip. Now we need 
herringbones whose spines connect arbitrary two points in 
the strip $\Oe$. If the spine goes from point $A$ to point 
$C$, we will denote the corresponding area by 
$\Om{HB}(A,C)$\footnote{The letters HB come from the word 
HerringBone.}. By analogy with complex numbers, we will 
denote by $\bar x$ the point $(x_1,-x_2)$, and also 
$x+a\df(x_1+a,x_2)$.

Using this notation, we state the main proposition of this 
section, which shows how two oppositely oriented herringbones 
can be joined within a common foliation.

\begin{Prop}
\label{091101}
Let there exist a point $C=(C_1,C_2)\in\Oe$\textup, 
such that we can construct\textup:
\begin{itemize}
\item a left herringbone $\Om{HB}(C+\eps,C^r)$\textup, and
\item a right herringbone $\Om{HB}(\bar C-\eps,C^l)$.
\end{itemize}
Then\textup, if we define in the rectangle between them 
with vertices at the points $(C_1+C_2,\eps)$\textup, 
$C+\eps$\textup, $(C_1-C_2,-\eps)$\textup, and $\bar C-\eps$ 
a bilinear function \textup(which is uniquely determined 
by its boundary values\textup)\textup, we obtain 
a diagonally concave $C^1$-smooth function.
\end{Prop}

\begin{proof}
In a general case a bilinear function has the following 
form:
\eq{121101}{
a(x_1^2-x_2^2)+bx_1+cx_2+d\,.
}
The coefficient will be found using the value of the function 
at four vertices:
\eq{121102}{
\begin{aligned}
f_+(C_1+C_2)=B(C_1+C_2,\eps)&=
a(C_1+C_2)^2-a\eps^2+b(C_1+C_2)+c\eps+d;
\\
A^l(C_1+\eps)\df B(C_1+\eps,C_2)&=
a(C_1+\eps)^2-aC_2^2+b(C_1+\eps)+cC_2+d;
\\
f_-(C_1-C_2)=B(C_1-C_2,-\eps)&=
a(C_1-C_2)^2-a\eps^2+b(C_1-C_2)-c\eps+d;
\\
A^r(C_1-\eps)\df B(C_1-\eps,-C_2)&=
a(C_1-\eps)^2-aC_2^2+b(C_1-\eps)-cC_2+d.
\end{aligned}
}
Solving this system we get
\eq{121103}{
\begin{aligned}
a&=\frac{A^r+A^l-f_+-f_-}{4(\eps^2-C_2^2)};
\\
b&=\frac{(C_1-C_2)f_++(C_1+C_2)f_-
-(C_1-\eps)A^l-(C_1+\eps)A^r}{2(\eps^2-C_2^2)};
\\
c&=\frac{\eps(f_+-f_-)+C_2(A^r-A^l)}{2(\eps^2-C_2^2)};
\\
d&=\frac{[\eps^2-(C_1-C_2)^2]f_+
+[\eps^2-(C_1+C_2)^2]f_-}{4(\eps^2-C_2^2)}+
\\
&\qquad\qquad\qquad+\frac{[(C_1-\eps)^2-C_2^2]A^l
+[(C_1+\eps)^2-C_2^2]A^r}{4(\eps^2-C_2^2)}.
\end{aligned}
}

The formulas will be a bit simpler, if we express $A^l$ and 
$A^r$ in terms of the boundary values $f_+$ и $f_-$ 
using~(4.16) for $A=A^l$ and~(4.44) for $A=A^r$:
$$
\begin{aligned}
A^l&=A(C_1+\eps,C_2)=
\\
&=\frac{\eps^2-C_2^2}{2\eps C_2}\Big[(\eps+C_2)f_+'-
(\eps-C_2)f_-'\Big]
+\frac{(\eps+C_2)f_++(\eps-C_2)f_-}{2\eps}\,;
\\
A^r&=A(C_1-\eps,-C_2)=
\\
&=\frac{\eps^2-C_2^2}{-2\eps C_2}\Big[(\eps+C_2)f_-'-
(\eps-C_2)f_+'\Big]
+\frac{(\eps-C_2)f_++(\eps+C_2)f_-}{2\eps}\,.
\end{aligned}
$$
Recall that the values of $f_+$ and $f_-$ are taken at the
vertices of the rectangle on the boundary: $f_+=f_+(C_1+C_2)$ 
and $f_-=f_-(C_1-C_2)$. After algebraic transformations
we get:
\begin{align}
a&=\frac{f_+'-f_-'}{4C_2};\label{181104}
\\
b&=\frac{(C_1+C_2)f_-'-(C_1-C_2)f_+'}{2C_2};\label{181105}
\\
c&=\frac{f_+-f_--C_2(f_+'+f_-')}{2\eps};
\\
d&=\frac{f_+\!\!+\!f_-\!}2+\frac{(C_1^2\!+
\!\eps^2\!\!-\!C_2^2\!-\!2C_1C_2)f_+'\!-\!
(C_1^2\!+\!\eps^2\!\!-\!C_2^2\!+\!2C_1C_2)f_-'}{4C_2}.
\end{align}

Due to the symmetry, it is sufficient to check continuity 
of the gradient on one arbitrary side of the rectangle.
Moreover, since the derivative along the side of the rectangle
is continuous, it is sufficient to check continuity along
any transversal direction. We will check continuity of
$B_{x_1}$ on the side connecting $(C_1+C_2,\eps)$ and $C+\eps$.

Formula~(4.14) supplies us with the derivative on a 
herringbone. We need the upper half of the formula that
gives us a linear function:
$$
B_{x_1}(x)=(\eps-x_2)N_+(u)+f_+'\,.
$$
The derivative in the rectangle
$$
B_{x_1}(x)=2ax_1+b
$$
is a linear function as well, therefore, for their coincidence
it is sufficient to check that they are equal at the endpoints 
of the interval:
\eq{181101}{
2a(C_1+C_2)+b=f_+'\qquad\text{at}\quad x=(C_1+C_2,\eps)
}
and
\eq{181102}{
2a(C_1+\eps)+b=(\eps-C_2)N_++f_+'\qquad\text{at}\quad 
x=(C_1+\eps,C_2).
}
The first equality immediately follows from~\eqref{181104} 
and~\eqref{181105}. To check the second equality we need
the definition of $N_+$ 
{\bf(}see~(4.11) and~(4.5){\bf)}. At the point
$(u,T)=(C_1+\eps,C_2)$ we get:
\eq{181106}{
N_+=\frac{\half R-f_+'}{\eps-C_2},
}
where
\eq{181107}{
R=R_++R_-,\qquad R_-=\frac{A^l-f_-}{\eps+C_2},\qquad 
R_+=\frac{A^l-f_+}{\eps-C_2}\,.
}
Thus, equality~\eqref{181102} can be rewritten in the form
\eq{181103}{
R=4a(C_1+\eps)+2b\,,
}
what is easier to verify without using explicit expressions 
for $a$ and $b$, but directly from equations~\eqref{121102}:
$$
\begin{aligned}
R_-=\frac{A^l-f_-}{\eps+C_2}=2a(C_1+\eps-C_2)+b+c;
\\
R_+=\frac{A^l-f_+}{\eps-C_2}=2a(C_1+\eps+C_2)+b-c,
\end{aligned}
$$
whence immediately~\eqref{181103} follows.
\end{proof}

The function constructed in this Proposition will be denoted 
by $\Om{Rect}(C)$.

\section{Examples. Polynomials of third degree: 
\\the case of a rectangle with two herringbones}

Recall that we were tracking the integral curves $\ell_\eps^+$ 
and $\ell_\eps^-$, which gave us the spines of the 
herringbones: the first, when shifted right by $\eps$, 
produced a left herringbone, while the second, when shifted 
left by $\eps$, produced a right herringbone. 
For convenience, instead of the curve $\ell_\eps^-$, 
we will track its reflection $\bar\ell_\eps^-$ across 
the $x_1$-axis. 

We are not interested in the behavior of 
these integral curves outside the strip $\Oe$, but inside 
the strip, they were disjoint for $\eps<\eps_2$. 
At $\eps = \eps_2$, the curves $\bar\ell_\eps^-$ and 
$\ell_\eps^+$ intersected at a point on the upper boundary. 
Let us denote their intersection point by $C = C(\eps)$. 
Then, $C_1(\eps_2) = v_+(\eps_2) = v_-(\eps_2)$ and 
$C_2(\eps_2) = \eps_2$.  

Both curves, $\ell_\eps^+$ and $\bar\ell_\eps^-$, descend 
as $\eps$ increases, so their intersection point $C$ also 
descends, i.\,e., the function $C_2(\eps)$ is strictly 
monotonically decreasing. Therefore, for all $\eps>\eps_2$, 
both the segment of $\ell_\eps^+$ between points $C$ and 
$U_+^r$ and the segment of $\bar\ell_\eps^-$ between points 
$\bar U_-^l$ and $C$ lie in the domain where $D_+ > 0$ and 
$D_- > 0$.  

Thus, the left herringbone constructed from $\ell_\eps^+$ 
and the right herringbone constructed from $\ell_\eps^-$ 
generate diagonally concave functions. 
{\bf(}Recall that $D_\pm\ut{R}(u)=D_\pm(u-\eps,-T\ti{R}(u))$, 
so the condition for the diagonal concavity of the 
corresponding function ($D_\pm\ut{R}\geq0$) can be 
rewritten as the condition $D_\pm\geq0$ on the curve 
$\bar\ell_\eps^-$.{\bf)}  

Hence, the constructed point $C$ satisfies 
Proposition~\ref{091101}, allowing us to define the foliation  
\eq{211101}{
\begin{aligned}
\Omega_\eps = \Om{R}(-\infty, u_-^l) &\cup 
\Om{HB}(\bar C-\eps, U_-^l-\eps)\cup\Om{Rect}(C)\cup 
\\
&\cup\Om{Rect}(C)\cup\Om{HB}(C+\eps,U_+^r+\eps)\cup
\Om{R}(u_+^r, +\infty),
\end{aligned}
}  
which generates the desired minimal diagonally concave 
function for all $\eps \geq \eps_2$.  

Recall that so far, we have assumed that 
inequality~\eqref{301101} holds. Now, we consider the case 
where this inequality does not hold. In this scenario, 
the foliation~\eqref{160403} persists up to $\eps_2$, 
which is now determined by the equality 
$v_+(\eps_2) = u_-^l(\eps_2)$. After this moment, the same 
foliation arises as in the case where condition~\eqref{301101} 
holds: two oppositely oriented herringbones separated by 
a rectangle.  

Let us explain this in more detail. We will track the curve 
$\bar\ell_\eps^-$ and verify that it lies strictly above 
$\ell_\eps^+$ for $\eps<\eps_2$. As we know, the curve 
$\bar\ell_\eps^-$ appears on the upper boundary at 
$\eps=\eps_0^-$ as the point $(u_0^-,\eps_0^-)$. 
Recall that we are now considering the case where 
inequality~\eqref{301101} does not hold. In the borderline 
case when $v_+(\eps_0^-) = u_0^-$, we set $\eps_2=\eps_0^-$. 
If $v_+(\eps_0^-)<u_0^-$, then the point $(u_0^-,\eps_0^-)$ 
lies strictly above the curve $\ell_\eps^+$, since the right 
endpoint $U^r_+$ of $\ell_\eps^+$ always lies to the right of 
this point {\bf(}cf.~formulas~\eqref{140301} 
and~\eqref{140303}{\bf)}.

Thus, let $\eps_2$ be the smallest value of $\eps$ for which 
the curves $\bar\ell_\eps^-$ and $\ell_\eps^+$ have a common 
point in the strip $\Oe$. We will show that this point must 
lie on the boundary of the strip. First, suppose this point 
lies strictly inside the strip. Then it would be a point of 
tangency between $\bar\ell_\eps^-$ and $\ell_\eps^+$. Let 
us equate the slopes of $\bar\ell_\eps^-$ and $\ell_\eps^+$ 
{\bf(}see~\eqref{070302} and~(5.2){\bf)}:
$$
-\frac{(\eps-x_2)D_+-(\eps+x_2)D_-}
{(\eps-x_2)D_++(\eps+x_2)D_-}=
\frac{(\eps-x_2)D_--(\eps+x_2)D_+}
{(\eps-x_2)D_-+(\eps+x_2)D_+}\,,
$$
which reduces to $x_2D_+D_-=0$. This cannot be fulfilled 
since our curves lie in the upper half of the strip ($x_2>0$) 
and we consider them in the domain where $D_+>0$ and $D_->0$.

Therefore, the curves $\bar\ell_{\eps_2}^-$ and 
$\ell_{\eps_2}^+$ can intersect only at points on 
the upper boundary $x_2=\eps_2$. This cannot be at the right 
endpoints of the curves, i.\,e., not at $U^r_+$. Indeed, 
the slope of $\bar\ell_{\eps_2}^-$ at this point equals 1, 
while the slope of $\ell_{\eps_2}^+$ is strictly less than 1. 
Hence, near $U^r_+$, the curve $\bar\ell_{\eps_2}^-$ would 
have to lie below $\ell_{\eps_2}^+$, which is impossible 
because by definition of $\eps_2$, for all $\eps<\eps_2$ 
the curve $\bar\ell_{\eps}^-$ lies strictly above 
$\ell_{\eps}^+$. Consequently, these curves must intersect 
at $\bar U^l_-=(v_+(\eps_2),\eps_2)$. Thus, as claimed, 
the parameter value $\eps_2$ is determined by the equation 
$v_+(\eps)=u^l_-(\eps)$. 

The solution to this equation is clearly unique (if it exists), 
since, as noted earlier, both curves $\bar\ell_{\eps}^-$ and 
$\ell_{\eps}^+$ descend as $\eps$ increases, and the second 
coordinate $C_2(\eps)$ of their intersection point $C=C(\eps)$ 
is strictly decreasing. If the equation $v_+(\eps)=u^l_-(\eps)$ 
has no solution, we may formally set $\eps_2=\infty$, in which 
case the foliation~\eqref{160403} persists for all 
$\eps > \eps_0^+$.

The foliation $\Om{Rect}(C)$ constructed in Section~\ref{211102} allows us to conjugate not only two pockets but also two oppositely directed fissures, as in the case described in Fig.~15 of \cite{First}.

Recall that in the case of a positive discriminant of the quadratic polynomial $f_+'-f_-'$ and opposite signs of the leading coefficients $a_3^\pm$ of the polynomials $f_\pm$ (we chose $a_3^+>0$ and $a_3^-<0$), the foliation for small $\eps$ is described by the formula
\begin{equation}\label{041201}
\begin{aligned}
\Oe=&\;\Om{R}(-\infty,u_{-1})\cup\Omega_\eps^{\scriptscriptstyle\mathrm{SW}}(u_{-1},v_{+1})\cup
\\
&\cup\Om{L}(v_{+1},v_{-2})\cup\Omega_\eps^{\scriptscriptstyle\mathrm{NE}}(v_{-2},u_{+2})\cup\Om{R}(u_{+2},+\infty)\,,
\end{aligned}
\end{equation}
This foliation persists as long as $v_{+1}<v_{-2}$. At the moment $\eps=\eps_2$, determined by the equality $v_{+1}(\eps_2)=v_{-2}(\eps_2)$, the domain $\Om{L}(v_{+1},v_{-2})$ collapses into a single left extremal line, after which a rectangle $\Om{Rect}(C)$ appears in its place. The point $C$ now rises from the lower boundary and is determined as follows. Let $\ell_\eps^+$ be the integral curve of field (5.3) starting from the point $(u_{+2},\eps)$ with positive slope (i.\,e., the curve generating the right fissure), and let $\ell_\eps^-$ be the integral curve of field~\eqref{070302} starting from the point $(u_{1-},-\eps)$ with positive slope (i.\,e., the curve generating the left fissure). Then the point $C=C(\eps)$ is the intersection point of the curves $\ell_\eps^+$ and $\bar\ell_\eps^-$.

Thus, we only need to verify that the functions $v_{+1}$ 
and $v_{-2}$ do not tend to finite limits. We know that 
the first function is monotonically increasing, while 
the second is monotonically decreasing. Moreover, 
the absolute values of the derivatives of these functions 
are not less than one. Indeed, consider Fig.~\ref{031201}.
\begin{figure}[h]
    \centering
    \includegraphics[scale = 0.6]{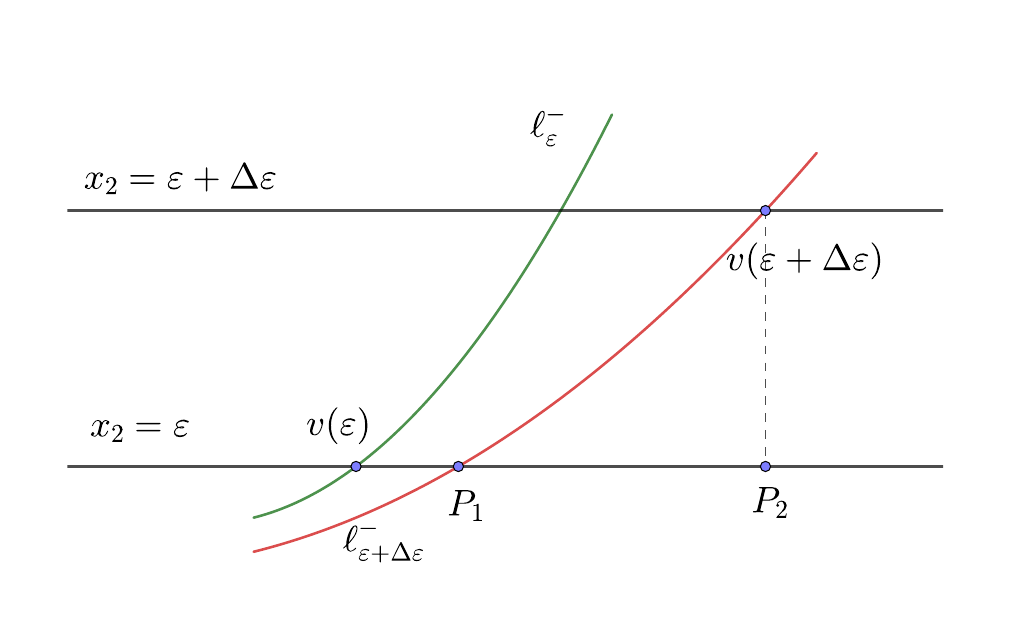}
    \caption{The curves $\ell_\eps^-$ and 
    $\ell_{\eps+\Delta\eps}^-$}
    \label{031201}
\end{figure}

\noindent Since the slope of any integral curve $\ell_\eps$ 
is strictly less than one inside the strip $\Oe$, the distance 
between points $P_1$ and $P_2$ is strictly bigger than 
$\Delta\eps$, and therefore $v(\eps+\Delta\eps)-
v(\eps)>\Delta\eps$, meaning $v'(\eps)\le1$. 
This reasoning applies not only to the case $v=v_{+1}$, 
where the arrangement of curves is as shown in the figure, 
but also to the symmetric case $v=v_{-2}$.

Thus, at $\eps=\eps_2$ a foliation restructuring occurs, 
and for all larger $\eps$ we obtain the foliation
\begin{equation}\label{041202}
\begin{aligned}
\Oe=\Om{R}(-\infty,u_{1-})&
\cup\Om{HB}(\bar C-\eps,(u_{1-}\!\!-\eps,-\eps))\:\cup
\\
\cup\:\Om{Rect}(C)&\cup\Om{HB}(C+\eps,(u_{2+}\!+\eps,\eps))
\cup\Om{R}(u_{2+},+\infty)\,,
\end{aligned}
\end{equation}
or in slightly different notation
\begin{equation}\label{041203}
\begin{aligned}
\Omega_\eps=\Om{R}(-\infty,u_-^l)&
\cup\Om{HB}(\bar C-\eps,U_-^l-\eps)\:\cup
\\
\cup\:\Om{Rect}(C)&\cup\Om{HB}(C+\eps,U_+^r+\eps)
\cup\Om{R}(u_+^r,+\infty)\,.
\end{aligned}
\end{equation}

Concluding this section, we describe one more case where 
a rectangle connecting two oppositely directed herringbones 
appears. This is the previously unmentioned case of a zero 
discriminant of the expression $f'_+-f'_-$. Until now we have 
assumed this expression has no multiple roots. We have just 
considered the case where $f'_+-f'_-$ has two roots with 
$a_3^+>0$ and $a_3^-<0$. If we bring the roots closer while 
keeping the leading coefficients of the polynomials, 
the critical strip width $\eps_2$ up to which 
foliation~\eqref{041201} holds will tend to zero, and in 
the limiting case of a multiple root there will be no 
fissures for any $\eps$. We will now show that in this case, 
foliation~\eqref{041203} holds for all $\eps$.

Thus, we consider the case when 
$$
(a_2^+-a_2^-)^2=3(a_3^+-a_3^-)(a_1^+-a_1^-), \qquad 
a_3^+>0,\quad\text{and}\quad a_3^-<0.
$$
Let's begin by examining the zeros of functions $D_+$ and 
$D_-$. In our case, formulas~\eqref{130301} and~\eqref{130302} 
can be written as follows:
\eq{080201}{
2x_2D_+(x)=3(a_3^+-a_3^-)\big(x_1-u_0-(1+2k_+)x_2\big)
\big(x_1-u_0-(1-2k_+)x_2\big);
}
\eq{080202}{
2x_2D_-(x)=3(a_3^+-a_3^-)\big(x_1-u_0+(1+2k_-)x_2\big)
\big(x_1-u_0+(1-2k_-)x_2\big),
}
where
\eq{080203}{
u_0=-\frac{a_2^+-a_2^-}{3(a_3^+-a_3^-)}\,,\quad 
k_+=\sqrt{\frac{a_3^+}{a_3^+-a_3^-}}\,,
\quad k_-=\sqrt{\frac{|a_3^-|}{a_3^+-a_3^-}}\,.
}

Thus, in this case the curves $D_+=0$ and $D_-=0$ are 
straight lines passing through the point $(u_0,0)$. 
As when constructing fissures, we'll be interested in 
the intersection points of lines $D_\pm=0$ with the boundary 
line $x_2=\eps$. Moreover, from each pair of points 
we'll need the right point for $D_+$
\eq{090201}{
(u_+,\eps),\qquad\text{where}\qquad u_+=u_0+(1+2k_+)\eps
}
and the left point for $D_-$
\eq{090202}{
(u_-,\eps),\qquad\text{where}\qquad u_-=u_0-(1+2k_-)\eps\,.
}

The point $(u_+,\eps)$ is a saddle point of field (5.3), 
while the point $(u_-,-\eps)$ is a saddle point of 
field~\eqref{070302}. The Jacobian matrix at $(u_+,\eps)$ 
equals (see (5.17))
\eq{120201}{
J(u_+,\eps)=12(a_3^+-a_3^-)k_+\eps^2
\begin{pmatrix}
1&-2-3k_+
\\
-1&k_+
\end{pmatrix}.
}

We're interested in the eigenvector of this matrix with 
positive slope. It can be verified that this will be the vector
\eq{120202}{
\begin{pmatrix}
k_+-1+\sqrt{(1+k_+)(9+k_+)}
\\
2
\end{pmatrix},
}
whose slope is greater than that of the line $x_1=(1+2k_+)x_2$ 
(on which, recall, $D_+=0$). Indeed, since $k_+>0$, it's easy 
to verify that
$$
\frac2{k_+-1+\sqrt{(1+k_+)(9+k_+)}}>\frac1{1+2k_+}\,.
$$

Thus, the integral curve $\ell_\eps^+$ of field~(5.3), 
passing through the point $(u_+,\eps)$ along 
vector~\eqref{120202}, enters the strip $\Oe$ below 
the line $x_1=u_0+(1+2k_+)x_2$. By the same reasoning used 
earlier, the curve $\ell_\eps^+$ cannot intersect this line 
at any point in the upper half of the strip, since at 
any point on this line the integral curves have a slope 
equal 1. This means that from the region below the line, 
integral curves intersect the line from left to right, 
while the right portion of curve~$\ell_\eps^+$, as we have 
established, lies below the line $x_1=u_0+(1+2k_+)x_2$. 
Therefore, curve~$\ell_\eps^+$ passes through the point 
$(u_0,0)$, since it cannot intersect the $x_1$-axis at any 
other point. The crucial conclusion for us is that in 
the upper half of strip $\Oe$, curve~$\ell_\eps^+$ lies 
to the right of the line $x_1=u_0+(1+2k_+)x_2$, i.\,e., 
in the region where $D_+>0$ and $D_->0$. Consequently, 
after shifting right by~$\eps$, it can serve as the spine 
of a left herringbone.

The construction of the right herringbone is completely 
symmetric, i.\,e., the construction of the integral curve 
$\ell_\eps^-$ of vector field~\eqref{070302}, which intersects 
the lower boundary at the saddle point $(u_-,-\eps)$ at a 
positive angle. The Jacobian matrix of field~\eqref{070302} 
has the form~\eqref{080403}. Evaluating the matrix at point 
$(u_-,-\eps)$, we obtain
\eq{150201}{
J(u_-,-\eps)=-12(a_3^+-a_3^-)k_-\eps^2
\begin{pmatrix}
1&-2-3k_-
\\
-1&k_-
\end{pmatrix}.
}
The eigenvector along which we launch curve $\ell_\eps^-$ 
has the form
\eq{150202}{
\begin{pmatrix}
k_--1+\sqrt{(1+k_-)(9+k_-)}
\\
2
\end{pmatrix},
}
Since
$$
\frac2{k_--1+\sqrt{(1+k_-)(9+k_-)}}>\frac1{1+2k_-}\,,
$$
the conjugate curve $\bar\ell_\eps^-$ starts from point 
$(u_-,\eps)$ below the line $x_1=u_0-(1+2k_-)x_2$, and 
thus remains below this line all the way to the point 
$(u_0,0)$, i.\,e., in the region where $D_+>0$ and $D_->0$. 
Therefore, after shifting left by~$\eps$, curve $\ell_\eps^-$ 
can serve as the spine of the right herringbone.

Thus, for any $\eps$, we are in the conditions of 
Proposition~\ref{091101} with point $C=(u_0,0)$, thereby 
constructing the desired foliation.

\section*{Acknowledgments}
The author is grateful to D.~M.~Stolyarov for conducting 
computer experiments that helped understand certain types 
of foliations, and to P.~B.~Zatitskii whose critical 
remarks helped improve this text.

\vskip40pt

\section*{Formulas and Propositions from~\cite{First}
used in this paper}

$$
B(x_1,x_2)=\frac{\eps+x_2}{2\eps}f_+(x_1-x_2+\eps)+
\frac{\eps-x_2}{2\eps}f_-(x_1-x_2-\eps)\,.
\leqno(2.1)
$$

$$
\begin{aligned}
f_+'(u+\eps)-f_-'(u-\eps)-2\eps f_+''(u+\eps)&\ge0\,;
\\
f_+'(u+\eps)-f_-'(u-\eps)-2\eps f_-''(u-\eps)&\ge0\,.
\end{aligned}
\leqno(2.2)
$$

\bigskip

{\bf Proposition 2.1.}
{\it The function $B$ defined by~$(2.1)$ is a Bellman
candidate on the domain $\Om{R}(u_1,u_2)$ if and only if 
conditions~$(2.2)$ are fulfilled for all
$u,$ $u\in(u_1,u_2)$.
}

\bigskip

$$
R_-\df\frac{A-f_-}{\eps+T},\qquad R_+\df
\frac{A-f_+}{\eps-T},\quad\text{and}\quad R\df R_-+R_+.
\leqno(4.5)
$$

$$
N_+\df\frac{\half R-f_+'}{\eps-T}\quad\text{and}\quad 
N_-\df\frac{\half R-f_-'}{\eps+T}\,.
\leqno(4.11)
$$

$$
B_{x_1}(x)=
\begin{cases}
(\eps-x_2)N_++f_+', 
&\text{if }x_2\ge T(u);
\\
(\eps+x_2)N_-+f_-',
&\text{if }x_2\le T(u).
\end{cases}
\leqno(4.14)
$$

$$
\begin{cases}
\dot x_1=(\eps-x_2)\big[\half(f_+'-f_-')-x_2f_-''\big]+
(\eps+x_2)\big[\half(f_+'-f_-')-x_2f_+''\big],
\\
\dot x_2=(\eps-x_2)\big[\half(f_+'-f_-')-x_2f_-''\big]-
(\eps+x_2)\big[\half(f_+'-f_-')-x_2f_+''\big],
\end{cases}
\leqno(5.2)
$$

\bigskip

$$
\begin{cases}
\dot x_1=\eps(f_+'-f_-')-x_2(\eps-x_2)f_-''-x_2(\eps+x_2)f_+'',
\\
\dot x_2=x_2(f_-'-f_+')-x_2(\eps-x_2)f_-''+x_2(\eps+x_2)f_+''.
\end{cases}
\leqno(5.3)
$$

\bigskip

{\bf Proposition 5.4.}
{\it The slope of an integral curve of field~$(5.2)$ at 
a point $x$ is strictly increasing with $\eps$ if 
$x_2D_+(x)D_-(x)>0$ and it is strictly decreasing if 
$x_2D_+(x)D_-(x)<0$.
}

\bigskip

$$
J(x)\!=\!
\begin{pmatrix}
\scriptstyle\eps(f_+''-f_-'')-x_2(\eps-x_2)f_-'''-
x_2(\eps+x_2)f_+'''&
\scriptstyle 2x_2(f_-''-f_+'')+x_2(\eps-x_2)f_-'''-
x_2(\eps+x_2)f_+'''
\\
\scriptstyle x_2(f_-''-f_+'')-x_2(\eps-x_2)f_-
'''+x_2(\eps+x_2)f_+'''&
\scriptstyle(f_-'-f_+')-(\eps-x_2)(f_-''-x_2f_-''')+
(\eps+x_2)(f_+''+x_2f_+''')
\end{pmatrix}.
\leqno(5.5)
$$

$$
J(u_+,\eps)=\eps\big(f_+''(u_++\eps)-f_-''(u_+-\eps)-
2\eps f_+'''(u_++\eps)\big)
\begin{pmatrix}
\phantom{i}1&-2-3\eps\kappa_+
\\
-1&\eps\kappa_+
\end{pmatrix},
\leqno(5.17)
$$

$$
\kappa_+=\frac{2f_+'''(u_++\eps)}{f_+''(u_++\eps)-
f_-''(u_+-\eps)-2\eps f_+'''(u_++\eps)}\,.
\leqno(5.18)
$$

\bigskip

St. Petersburg State University, 

St. Petersburg, Russia

vasyunin@pdmi.ras.ru
\end{document}